\documentclass[12pt]{amsart}
\usepackage{color}
\usepackage{a4wide}
\usepackage{verbatim}

\usepackage{eucal}

\begin{document}
\newtheorem{theorem}{Theorem}
\newtheorem{proposition}[theorem]{Proposition}
\newtheorem*{proposition*}{Proposition}
\newtheorem{lemma}[theorem]{Lemma}
\theoremstyle{definition}
\newtheorem{definition}[theorem]{Definition}
\newtheorem*{definition*}{Definition}
\newtheorem{example}{Example}
\theoremstyle{remark}
\newtheorem*{remark*}{Remark}

\newcommand{\Ll}{\operatorname{\mathcal{L}}}
\newcommand{\cntrct}{\hspace{2pt}\raisebox{1pt}{\text{$\lrcorner$}}\hspace{2pt}}
\newcommand{\End}{\operatorname{End}}
\newcommand{\Tr}{\operatorname{Tr}}
\newcommand{\al}{\alpha}
\newcommand{\be}{\beta}
\newcommand{\vol}{\operatorname{vol}}
\newcommand{\const}{\operatorname{\text{\sf const}}}
\newcommand{\og}{\overline{g}}
\newcommand{\oOmega}{\overline{\omega}}
\newcommand{\otheta}{\overline{\theta}}
\newcommand{\oT}{\overline{T}}

\title{Locally conformally K\"ahler manifolds with holomorphic Lee field}
\author{Andrei Moroianu}
\address{Andrei Moroianu, Laboratoire de Math\'ematiques d'Orsay, Univ.\ Paris-Sud, CNRS,
Universit\'e Paris-Saclay, 91405 Orsay, France}
\email{andrei.moroianu@math.cnrs.fr}
\author{Sergiu Moroianu}
\address{Sergiu Moroianu, Institutul de Matematic\u{a} al Academiei Rom\^{a}ne\\
P.O. Box 1-764\\RO-014700
Bu\-cha\-rest, Romania}
\email{moroianu@alum.mit.edu}
\author{Liviu Ornea}
\address{Liviu Ornea, University of Bucharest, Faculty of Mathematics\\14
Academiei str.\\70109 Bucharest, Romania\\and:
Institute of Mathematics ``Simion Stoilow'' of the Romanian
Academy\\
21 Calea Grivitei\\
010702-Bucharest, Romania}
\email{lornea@fmi.unibuc.ro, Liviu.Ornea@imar.ro}
\thanks{Sergiu Moroianu was partially supported by the CNCS - UEFISCDI grant PN-III-P4-ID-PCE-2016-0330. Liviu Ornea was partially supported by the CNCS - UEFISCDI grant PN-III-P4-ID-PCE-2016-0065}
\begin{abstract}
We prove that a compact lcK manifold with holomorphic Lee vector field is Vaisman provided that either the Lee field has constant norm or the metric is Gauduchon (i.e., the Lee field is divergence-free). We also give examples of compact lcK manifolds with holomorphic Lee vector field which are not Vaisman.
\end{abstract}

\maketitle

\section{Introduction}

Let $(M,J,g)$ be a Hermitian manifold of complex dimension $n\geq 2$. Let $\omega$ be its fundamental form defined as $\omega(\cdot,\cdot)=g(J\cdot, \cdot)$.

\begin{definition*}
A Hermitian manifold $(M,J,g)$ is called {\em locally conformally K\"ahler} (lcK) if there exists a closed 
$1$-form $\theta$ which satisfies
\[d\omega=\theta\wedge\omega.\]
If moreover $\theta$ is parallel with respect to the Levi-Civita connection of $g$, the manifold (and the metric itself) are called {\em Vaisman} \cite{va}.
\end{definition*}  

We refer to \cite{do} for definitions and main properties of lcK metrics. 

Among lcK manifolds, Vaisman ones are important because their topology is better understood. Consequently, characterizations of this subclass are always interesting.
Several sufficient conditions are known for a compact, non-K\"ahler  lcK manifold to be Vaisman: the Einstein-Weyl condition (\cite{pps}); the existence of a parallel vector field (\cite{mor}); the pluricanonical condition (\cite{mm}); homogeneity (\cite{gmo}); having potential and being embedded in a diagonal Hopf manifold (\cite{ov_lckpot}, \cite{ov_indam}); or being toric (\cite{is}, \cite{mmp}).  

It is known that on a Vaisman manifold, the Lee and anti-Lee vector fields $\theta^\sharp$ and $J\theta^\sharp$ are holomorphic and Killing. In this note, we discuss the implications of the Lee vector field being only holomorphic, and we add to the above list of sufficient conditions the following result:

\begin{theorem}\label{main} Let $(M,J,g,\theta)$ be a compact locally conformally K\"ahler manifold with holomorphic Lee field $\theta^\sharp$. Suppose that one of the following conditions is satisfied:
\begin{enumerate}
\item[(i)] The norm of the Lee form $\theta$ is constant, {\em or}
\item[(ii)] The metric $g$ is Gauduchon (which means by definition that the Lee form $\theta$ is co-closed with respect to $g$, see \cite[pp.\ 502]{gau}).
\end{enumerate}
Then $(M,J,g)$ is Vaisman.
\end{theorem}

As shown in Section 3, this result does not hold in general for holomorphic Lee vector fields $\theta^\sharp$  without imposing some additional hypotheses like (i) or (ii).

Very recently, Nicolina Istrati has found another instance where the conclusion of Theorem \ref{main} holds:

\begin{proposition*}[Istrati \cite{ni}] Let $(M,J,g,\theta)$ be a compact locally conformally K\"ahler manifold with holomorphic Lee field. Suppose that the metric $g$ has a potential, i.e., $\omega=\theta\wedge J\theta-dJ\theta$ (cf. \cite{ov_lckpot}).
Then $(M,J,g)$ is Vaisman.
\end{proposition*}

\section{Proof of Theorem \ref{main}}
\subsection{Local formulae}

Consider the Lee vector field $T:=\theta^\sharp$. We start by proving two straightforward results.

\begin{lemma}\label{unu} If $T$ is holomorphic, then $JT$ is both holomorphic and Killing.
\end{lemma}
\begin{proof}
Since $J$ is integrable and $T$ is holomorphic, $JT$ must also be holomorphic. On any lcK manifold, Cartan's formula yields:
\begin{equation*}
\begin{split}
\Ll_{JT}\omega&=d(JT\cntrct \omega)+JT\cntrct(\theta\wedge\omega)\\
&=-d\theta+(\theta\wedge\theta)=0.
\end{split}
\end{equation*}
Combined with $\Ll_{JT}J=0$, this gives $\Ll_{JT}g=0$.
\end{proof}
\begin{lemma}\label{T}
The Lee vector field $T$ is holomorphic if and only if the symmetric endomorphism $F:=\nabla T\in \End(TM)$ commutes with $J$.
\end{lemma}
\begin{proof} On any lcK manifold one has the well-known identity (see e.g.\ \cite[Cor.\ 1.1]{do})
\begin{equation}\label{naj}
\nabla_XJ=\tfrac{1}{2}(X\wedge J\theta+JX\wedge\theta).
\end{equation}
In particular, $\nabla_T J=\tfrac{1}{2}(T\wedge J\theta+JT\wedge\theta)=0$.
We thus get for every vector field $X$ on $M$:
\begin{equation*}
\begin{split}
(\Ll_TJ)(X)&=[T,JX]-J[T,X]=\nabla_T JX-\nabla_{JX}T-J\nabla_TX+J\nabla_XT\\
&=(\nabla_TJ)(X)-\nabla_{JX}T+J\nabla_XT=-F(JX)+J(FX),
\end{split}
\end{equation*}
thus proving our claim.
\end{proof}

For later use, we remark that \eqref{naj} yields for every local orthonormal frame $\{e_i\}$:
\begin{align}\label{cgnt}
\sum_i(\nabla_{e_i}J)(e_i)=(n-1)JT,&&
\sum_i(\nabla_{Je_i}J)(e_i)=-(n-1)T.
\end{align}

Since $\theta$ is closed, the symmetric bilinear form defined as $S(X,Y):=(\nabla_X\theta)(Y)=g(FX,Y)$ measures the failure of $T$ to being Killing:
\begin{equation}\label{doi}
\Ll_Tg=2S.
\end{equation}

Note that by Lemma \ref{T}, $T$ is holomorphic if and only if $S$ is of type $(1,1)$, i.e. $(\nabla\theta)^{2,0}+(\nabla\theta)^{0,2}=0$. In fact, this is precisely the complementary condition to the lcK structure being pluricanonical, which reads 
$(\nabla\theta)^{1,1}=0$, see \cite{mm}.

We now derive a formula which is the main tool for proving Theorem \ref{main}:

\begin{proposition}\label{trei} If $T$ is holomorphic, then:
\begin{equation}\label{djd}
\begin{split}
dJd|\theta|^2={}&4\omega(F^2\cdot,\cdot)+2\omega(\Ll_TF\cdot,\cdot)-T(|\theta|^2)\omega
-2|\theta|^2\omega(F\cdot,\cdot)\\
&+d|\theta|^2\wedge J\theta +\theta\wedge Jd|\theta|^2.
\end{split}
\end{equation}
\end{proposition}
\begin{proof} 
By \eqref{doi} one has:
\begin{equation}\label{e3}
\Ll_T\omega=\Ll_T(g(J\cdot,\cdot))=2S(J\cdot,\cdot)=2\omega(F\cdot,\cdot).
\end{equation}
Using this we compute:
\begin{equation}\label{e4}
\begin{split}
d(J\theta)&=d(T\cntrct\omega)=\Ll_T\omega-T\cntrct d\omega\\
&=2\omega(F\cdot,\cdot)-|\theta|^2\omega+\theta\wedge J\theta.
\end{split}
\end{equation}
 Now, as $T$ is holomorphic, Cartan's formula implies:
$$\Ll_T(J\theta)=J\Ll_T\theta=J(d(T\cntrct\theta)+T\cntrct d\theta)=Jd|\theta|^2,$$
whence
$$dJd|\theta|^2=d\Ll_T(J\theta)=\Ll_T(d(J\theta)),$$
which taking \eqref{e3} and \eqref{e4} into account completes the proof.
\end{proof}

 Define now the trace of a two-form $\eta$ with respect to $\omega$ by
$$\Tr_\omega\eta:=\sum_i\eta(e_i, Je_i),$$
where $\{e_i\}_{1\leq i\leq 2n}$ is a local orthonormal frame (of course, the trace does not depend on the choice of the frame). The properties of $\Tr_\omega$ are summarized in:

\begin{lemma}\label{patru}
\begin{enumerate}
\item[(i)] If $A\in \End(TM)$ is the endomorphism associated a $2$-form $\eta$ by the requirement 
$\eta=\omega(A\cdot,\cdot)$, then
\[\Tr_\omega\eta=\Tr A:=\sum_i g(Ae_i, e_i).\]
\item[(ii)] If $\al$ and $\be$ are $1$-forms, then
\[\Tr_\omega(\al\wedge\be)=2g(J\al,\be).\]
\item[(iii)] If $f$ is a smooth function, then:
\[\Tr_\omega(dJdf)=-2\Delta f + 2(1-n)T(f).\]
\end{enumerate}
\end{lemma}
\begin{proof} The statements (i) and (ii) are straightforward. In order to prove (iii), we use \eqref{cgnt} and write in a local orthonormal frame parallel at a point where the computation is done:
\begin{align*}
\Tr_\omega(dJdf)&=\sum_ie_i(Jdf(Je_i))-\sum_i Je_i(Jdf(e_i))-\sum_i Jdf([e_i, Je_i])\\
&=\sum_i e_i(e_i(f))+Je_i(Je_i(f))- Jdf\left(\sum_i(\nabla_{e_i}J)e_i\right)\\
&=-\Delta f-\Delta f+\sum _i(\nabla_{Je_i}Je_i)(f)-Jdf((n-1)JT)\\
&=-2\Delta f + 2(1-n)T(f).
\end{align*}
\end{proof}
We now apply $\Tr_\omega$ to both sides of equation \eqref{djd} using Lemma \ref{patru} 
with $f=|\theta|^2$ and $A=F$:
\begin{align*}
-2\Delta |\theta|^2 + 2(1-n)T(|\theta|^2)={}&4\Tr F^2+2T(\Tr F)-2nT(|\theta|^2)\\
&-2|\theta|^2\Tr F+2T(|\theta|^2)+2T(|\theta|^2).
\end{align*}
But $\Tr F=-\delta \theta$ and $\Tr F^2=|\nabla \theta|^2$, hence we obtain:
\begin{equation}\label{cinci}
\Delta|\theta|^2+T(|\theta|^2)+|\theta|^2\delta\theta+2|\nabla \theta|^2-T(\delta\theta)=0.
\end{equation}

\subsection{Global consequences}
From now on, suppose $M$ is compact and that one of the assumptions below holds:
\begin{enumerate}
\item[{\bf A.}] Assume that $|\theta|$ is constant. Let then $x_0$ be a point where $\delta\theta$ attains its maximum. As $\int_M \delta\theta \vol_g=0$, we see that $(\delta\theta)(x_0)\geq 0$.

On the other hand, the first and second terms in \eqref{cinci} vanish because of the assumption $|\theta|=\const$, and $T(\delta\theta)(x_0)=0$ since $x_0$ is an extremum of $\delta\theta$. Then \eqref{cinci} implies:
$$|\theta|^2(x_0)(\delta\theta)(x_0)+2|\nabla \theta|^2(x_0)=0.$$
As both terms are positive, we have, in particular $(\delta\theta)(x_0)=0$ and so $\delta\theta\leq 0$. Since the integral of $\delta\theta$ vanishes, this means that $\delta\theta$ vanishes identically on $M$. Finally, \eqref{cinci} now reduces to $\nabla\theta=0$, i.e., $g$ has parallel Lee form, and Theorem \ref{main}(i) is proven.
\item[{\bf B.}]  Assume that $\delta\theta=0$. We integrate \eqref{cinci} on $M$:
\begin{equation*}
\begin{split}
0&=\int_M (\Delta|\theta|^2+T(|\theta|^2)+2|\nabla \theta|^2)\vol_g\\
&=\int_M (|\theta|^2\delta\theta +2|\nabla \theta|^2)\vol_g\\
&=2\int_M |\nabla \theta|^2\vol_g,
\end{split}
\end{equation*}
which again implies $\nabla\theta=0$ and proves Theorem \ref{main}(ii).
\end{enumerate}

\section{Counterexample under weaker hypotheses}

We now show that the conclusion of Theorem \ref{main} does not hold in general under the unique assumption that the Lee field is holomorphic. For this, consider a (compact) Vaisman manifold 
$(M,g,J,\omega, \theta)$ with
\begin{align*}
\omega=g(J\cdot,\cdot),&&d\omega=\theta\wedge\omega,&&|\theta|=1.
\end{align*}

\begin{lemma}\label{lema}Let $f\in C^\infty(M)$ be a non-constant function whose gradient is collinear to $T$, or equivalently
\begin{equation}\label{dft}
df\wedge\theta=0, 
\end{equation}
and such that $f>-1$. Define 
\begin{align}\label{eoO}
\oOmega:=\omega+f\theta\wedge J\theta, &&\og:=\oOmega(\cdot,J\cdot).
\end{align}
Then $(M,\bar g,J,\bar\omega)$ is lcK, with Lee form $\otheta=(1+f)\theta$ and Lee vector field $\bar T=T$.
\end{lemma}
\begin{proof}
Recall that on a Vaisman manifold the Lee vector field is both Killing and holomorphic, thus
$\Ll_T\omega=0$. This can be rewritten as
\[0=dJ\theta+T\cntrct(\theta\wedge \omega)=dJ\theta+\omega-\theta\wedge J\theta
\]
or equivalently
\begin{equation}\label{omegavais}
\omega=\theta\wedge J\theta-dJ\theta.
\end{equation}

The condition $f>-1$ ensures that $\og$ is positive-definite. We compute using \eqref{dft} and \eqref{omegavais}:
\begin{align*}
d\oOmega&=\theta\wedge\omega+df\wedge\theta\wedge J\theta-f\theta\wedge dJ\theta\\
&=\theta\wedge\omega+f\theta\wedge\omega\\
&=(1+f)\theta\wedge\oOmega.
\end{align*}
Hence $(M,\bar g,J,\bar\omega)$ is lcK, with Lee form $\otheta=(1+f)\theta$. If we denote by $\oT$ the associated Lee vector field, from \eqref{eoO} we have
\[T\cntrct\oOmega=J\theta+fJ\theta=(1+f)J\theta= J\otheta=\oT\cntrct \oOmega
\]
therefore $\oT=T$ as claimed. 
\end{proof}

Summarizing, the Lee vector field $\oT$ of the lcK manifold $(M,\bar g,J,\bar\omega)$ is holomorphic, however $\og$ is not Vaisman since the norm of $\oT$ is not constant:
\[\og(\oT,\oT)=\otheta(\oT)=(1+f)\theta(T)=1+f.
\]
In conclusion, the holomorphy of the Lee field $T$ alone is not enough for an lcK metric to be Vaisman.

An example of a Vaisman manifold and a function $f$ satisfying the hypotheses of Lemma \ref{lema} can be easily constructed as follows.
Let $M=S^1\times S^{2n-1}$ be the Hopf manifold obtained as the quotient of $\mathbb{C}^n\setminus \{0\}$ by a cyclic group of dilations generated by $x\mapsto ax$ for some $a>1$. The metric $4 r^{-2} g_{\mathbb{C}^n}$ and the standard complex structure on $\mathbb{C}^n$ are invariant by dilations, hence they descend to a Vaisman metric on $M$ with parallel Lee form $\theta =-2dr/r$ of norm $1$. Every non-constant function in the variable $r$ bounded by $1$ in absolute value, and verifying $f(r)=f(ar)$ for every $r$, for instance $f(t)=\tfrac{1}{2}\sin(2\pi\ln(r)/\ln(a))$, projects to a function on $M$ verifying the hypotheses of Lemma \ref{eoO}, and thus induces a non-Vaisman lcK metric on $M$ with holomorphic Lee vector field.
\begin{remark*}
Given any closed form $\theta$ on a connected smooth manifold $M$, there exists a non-constant function $f$ such that $df\wedge\theta=0$  if and only if the line spanned by the cohomology class $[\theta]\in H^1(M,\mathbb{R})$ contains an integer class. 
\end{remark*}


\end{document}